\documentclass[11pt]{amsart}

\usepackage[a4paper,margin=1in]{geometry}
\usepackage{amsmath,amssymb,amsthm,mathtools}
\numberwithin{equation}{section}
\usepackage{bm}
\usepackage{microtype}
\usepackage{xcolor}
\usepackage[colorlinks=true,linkcolor=blue!45!black,citecolor=blue!45!black,urlcolor=blue!55!black]{hyperref}

\newcommand{\R}{\mathbb R}
\newcommand{\dd}{\,\mathrm d}
\newcommand{\cK}{\mathcal K}
\newcommand{\cH}{\mathcal H}
\newcommand{\Tr}{\operatorname{Tr}}

\newtheorem{theorem}{Theorem}[section]
\newtheorem{proposition}[theorem]{Proposition}
\newtheorem{lemma}[theorem]{Lemma}
\newtheorem{corollary}[theorem]{Corollary}
\theoremstyle{remark}
\newtheorem{remark}[theorem]{Remark}

\title{\bfseries A Quantitative P\'olya--Szeg\H{o} Theorem for Tangential Polygons}
\author{Changfeng Gui}
\address{Department  of  Mathematics,  University  of  Macau,  Macau  SAR,  P. R. China}
\address{Zhuhai UM Science and Technology Research Institute, Hengqin, Guangdong, 519031,  P. R. China}
\email{changfenggui@um.edu.mo}
\author{Yeyao Hu}
\address{School of Mathematics and Statistics, HNP-LAMA, Central South University, Changsha, Hunan 410083, P. R. China}
\email{huyeyao@gmail.com}
\author{Qinfeng Li}
\address{School of Mathematics, Hunan University, Changsha, P.R. China.}
\email{liqinfeng1989@gmail.com}
\date{}

\begin{document}
	
	\begin{abstract}
		For a bounded Lipschitz domain $\Omega\subset\mathbb R^2$, let
		$T(\Omega)=\int_\Omega u_\Omega\,\dd x$ denote its torsional rigidity,
		where $-\Delta u_\Omega=1$ in $\Omega$ and $u_\Omega=0$ on $\partial\Omega$.
		We prove a quantitative P\'olya--Szeg\H{o} inequality for tangential polygons.
		Let $N\ge3$, let $P$ be a tangential $N$-gon, set $A=|P|$, and let
		$R_N$ be the regular $N$-gon of area $A$.
		Writing $L(\cdot)$ for perimeter, we obtain the explicit deficit estimate
		\[
		T(R_N)-T(P)\ge
		\frac{A^2}{8N\tan(\pi/N)}
		\left(1-\frac{L(R_N)^2}{L(P)^2}\right)^2.
		\]
		Thus, at fixed area, the torsional deficit is controlled from below purely
		by the perimeter ratio. In particular, the regular $N$-gon is the unique
		maximizer of torsional rigidity among tangential $N$-gons of prescribed area;
		for $N=3$ this gives the classical triangular P\'olya--Szeg\H{o} theorem with
		a quantitative estimate. The same perimeter estimate yields an explicit
		positive lower bound for $T(R_{N+1})-T(R_N)$ for equal-area regular polygons,
		and hence a rather short alternative proof of the strict monotonicity of torsional
		rigidity in $N$. Combined with the Kohler--Jobin inequality, it also gives an
		explicit sufficient condition for the polygonal Faber--Krahn inequality within
		the tangential class. Our full quantitative inequality is stronger: it contains,
		in addition, a nonnegative angular Jensen deficit, which yields quantitative
		angular stability away from degenerate configurations.
	\end{abstract}
	
	\maketitle
	
	\section{Introduction}
	
	Let $\Omega\subset\R^2$ be a bounded Lipschitz domain. Its Dirichlet torsion function solves
	\begin{equation}\label{eq:intro-torsion}
		-\Delta u_\Omega=1\quad\text{in }\Omega,
		\qquad
		u_\Omega=0\quad\text{on }\partial\Omega,
	\end{equation}
	and the torsional rigidity is
	\begin{equation}\label{eq:intro-T}
		T(\Omega)=\int_\Omega u_\Omega\,\dd x.
	\end{equation}
	For a polygon $P$, we write $L(P)$ for its perimeter. Equivalently,
	\begin{equation}\label{eq:variational-intro}
		T(\Omega)=\max_{v\in H_0^1(\Omega)}
		\left\{2\int_\Omega v\,\dd x-\int_\Omega|\nabla v|^2\,\dd x\right\}.
	\end{equation}
	The Saint--Venant inequality states that the disk uniquely maximizes $T$ among planar domains of prescribed area; see, for instance, P\'olya--Szeg\H{o} \cite{PolyaSzego1951} and Henrot \cite{HenrotBook}. Its polygonal analogue asks whether the regular $N$-gon uniquely maximizes torsional rigidity among all $N$-gons of prescribed area. Classical symmetrization proves this for $N=3,4$ \cite{Polya1948,PolyaSzego1951}, while the problem remains open for $N\ge5$; a key obstruction is that Steiner symmetrization need not preserve the number of sides. For comparison, the analogous fixed-area polygonal isoperimetric problem for logarithmic capacity was settled by Solynin and Zalgaller, who proved that the regular $N$-gon minimizes logarithmic capacity among all $N$-gons of prescribed area \cite{SolyninZalgaller2004}. Sharp quantitative stability for the classical Saint--Venant inequality, with the disk as comparison shape, is contained in the work of Brasco, De Philippis, and Velichkov \cite{BrascoDePhilippisVelichkov2015}; in particular, the torsional deficit controls the square of the Fraenkel asymmetry. The fixed-number-of-sides problem has a different geometry, with the regular polygon replacing the disk as the comparison shape.
	
	The polygonal Saint--Venant problem has also been considered directly in recent shape-optimization work. Bogosel and Bucur discuss the corresponding torsional-rigidity conjecture in \cite[Remark~7.8]{BogoselBucur2024}. More recently, D\'iaz-Avalos and Laurain developed a shape-calculus framework for simple polytopes with a fixed number of facets and applied it to a numerical study of the polygonal and polyhedral Saint--Venant problem; in two dimensions their computations support the regular-polygon conjecture \cite{DiazAvalosLaurain2026}.
	
	We solve the fixed-area problem in the class of tangential polygons. A polygon is \emph{tangential}, or \emph{circumscribed}, if it is circumscribed about a circle. Every triangle is tangential, and for $N\ge4$ the class contains many nonsymmetric polygons. The common incircle provides a canonical decomposition into right triangles, making this class both geometrically flexible and analytically tractable. Solynin proved that, among $N$-gons with prescribed inradius, the regular $N$-gon minimizes torsional rigidity \cite{Solynin1993}, and Solynin and Zalgaller extended the corresponding inradius inequalities to curvilinear polygons \cite{SolyninZalgaller2010}. Fragal\`a, Gazzola, and Lamboley obtained sharp $p$-torsion estimates and verified the fixed-area polygonal P\'olya--Szeg\H{o} conjecture on a subclass selected by a quantitative asymmetry condition \cite{FragalaGazzolaLamboley2013}. Bucur and Fragal\`a proved regular-polygon optimality for several related constrained variational problems involving torsional rigidity \cite{BucurFragala2021}, while Keady obtained lower bounds specific to tangential and regular polygons \cite{Keady2021,Keady2022}. Taken together, these polygonal results address different constraints, partial subclasses, one-sided estimates, or numerical optimization; none gives a rigorous global fixed-area maximization over the entire class of tangential $N$-gons, nor the quantitative deficit obtained here.
	
	Our main result is quantitative. For a tangential $N$-gon $P$ with inradius $r$, write its interior angles as $\pi-2\alpha_i$, so that
	\[
	0<\alpha_i<\frac\pi2,
	\qquad
	\sum_{i=1}^N\alpha_i=\pi,
	\]
	and set
	\[
	a_i=\tan\alpha_i,
	\qquad
	a_*=\tan\frac\pi N.
	\]
	For $a>0$, let
	\[
	D_a=\{(x,y):0<x<1,\ 0<y<ax\},
	\]
	and denote by $h(a)$ the variational maximum on $D_a$ with Dirichlet condition on the vertical side and homogeneous Neumann conditions on the other two sides; equivalently, $h$ is defined by \eqref{eq:hdef} below.
	
	\begin{theorem}[Quantitative P\'olya--Szeg\H{o} inequality]\label{thm:main}
		Let $N\ge3$, let $P$ be a tangential $N$-gon with $r$ the inradius, and let $R_N$ be the regular $N$-gon with $|R_N|=|P|$. Then
		\begin{equation}\label{eq:main-deficit}
			\begin{aligned}
				T(R_N)-T(P)\ge 2r^4\Bigg[&
				N\left(h(a_*)-\frac{a_*}{8}\right)
				-\sum_{i=1}^N\left(h(a_i)-\frac{a_i}{8}\right)\\
				&+\frac{Na_*}{16}
				\left(\frac{L(P)^2}{L(R_N)^2}-1\right)^2
				\Bigg].
			\end{aligned}
		\end{equation}
		The angular deficit and the perimeter term on the right are nonnegative; in fact, each vanishes if and only if $P$ is regular.
		Moreover, if all $\alpha_i$ belong to a compact interval $J\Subset(0,\pi/2)$ containing $\pi/N$, then there exists $c_J>0$ such that
		\begin{equation}\label{eq:main-stability}
			T(R_N)-T(P)
			\ge2r^4\left[
			\frac{c_J}{2}\sum_{i=1}^N\left(\alpha_i-\frac\pi N\right)^2
			+\frac{Na_*}{16}\left(\frac{L(P)^2}{L(R_N)^2}-1\right)^2
			\right].
		\end{equation}
	\end{theorem}
	
	The sign of the angular term is completely explicit. Indeed, by the strict concavity of 
	\[
	\alpha \mapsto h(\tan \alpha)-\frac{1}{8}\tan \alpha
	\]
	proved in Theorem \ref{thm:concavity} and $\sum_i\alpha_i=\pi$, Jensen's inequality gives
	\begin{equation}\label{eq:intro-angular-deficit}
		\sum_{i=1}^N\left(h(a_i)-\frac{a_i}{8}\right)
		\le
		N\left(h(a_*)-\frac{a_*}{8}\right),
	\end{equation}
	with equality if and only if $\alpha_i=\pi/N$ for every $i$. Thus the first difference in \eqref{eq:main-deficit} is a genuine nonnegative angular deficit.
	
	Discarding this angular deficit gives a particularly transparent consequence of Theorem \ref{thm:main}.
	\begin{corollary}\label{cor:perimeter-deficit}
		Under the assumptions of Theorem \ref{thm:main},
		\begin{equation}\label{eq:perimeter-deficit-r}
			T(R_N)-T(P)
			\ge \frac{Na_*r^4}{8}
			\left(\frac{L(P)^2}{L(R_N)^2}-1\right)^2.
		\end{equation}
		Equivalently, if $A=|P|=|R_N|$, then
		\begin{equation}\label{eq:perimeter-deficit-area}
			T(R_N)-T(P)
			\ge
			\frac{A^2}{8N\tan(\pi/N)}
			\left(1-\frac{L(R_N)^2}{L(P)^2}\right)^2.
		\end{equation}
		In either form, the right-hand side is positive unless $P$ is regular.
	\end{corollary}
	The second form contains no torsion function or auxiliary cell quantity: once the area is fixed, the lower bound depends only on the perimeter ratio. In particular, the unique-maximizer statement below already follows from this perimeter deficit alone.
	
	\begin{corollary}[P\'olya--Szeg\H{o} theorem for torsional rigidity over tangential polygons]\label{cor:extremal}
		Among tangential $N$-gons of prescribed area, the regular $N$-gon uniquely maximizes torsional rigidity. In particular, among all triangles of prescribed area, the equilateral triangle is the unique maximizer.
	\end{corollary}
	
	The perimeter deficit has the following consequence that is invisible in the qualitative inequality. Let $R_N$ and $R_{N+1}$ be regular polygons of the same area $A$, and let $r_N$ be the inradius of $R_N$.
	
	\begin{corollary}\label{cor:regular-monotonicity}
		For every $N\ge3$,
		\begin{equation}\label{eq:intro-monotonicity}
			\begin{aligned}
				T(R_{N+1})-T(R_N)
				\ge{}&\frac{r_N^4(N+1)\tan(\pi/(N+1))}{8}\\
				&\times\left(
				\frac{N\tan(\pi/N)}{(N+1)\tan(\pi/(N+1))}-1
				\right)^2>0.
			\end{aligned}
		\end{equation}
	\end{corollary}
	This gives a really short alternative proof of the strict monotonicity previously obtained in \cite{GuiHuLiZhang2026} by a Schwarz--Christoffel/Bergman analytic-content series argument. The proof here avoids the rather complicated computations in \cite{GuiHuLiZhang2026}.
	
	The same quantitative deficit also interacts naturally with spectral inequalities. The polygonal Faber--Krahn conjecture is known for $N=3,4$ and remains open for $N\ge5$; see \cite{HenrotBook,LaugesenSiudeja2017,BogoselBucur2024}. Recent work has established local minimality of the regular pentagon and hexagon by validated computation \cite{BogoselBucurValidated2024}. In Section \ref{sec:eigenvalue}, we combine \eqref{eq:main-deficit} with the Kohler--Jobin inequality to obtain an explicit sufficient condition for
	\[
	\lambda_1(P)>\lambda_1(R_N).
	\]
	As $N\to\infty$, this confines every possible counterexample within the tangential class to a relative perimeter layer $L(P)/L(R_N)-1=O(N^{-2})$. For context, the monotonicity of $\lambda_1(R_N)$ itself was conjectured by Antunes and Freitas \cite{AntunesFreitas2006}; Nitsch obtained comparison estimates \cite{Nitsch2014}, and Berghaus--Georgiev--Monien--Radchenko derived high-order asymptotic expansions \cite{BerghausGeorgievMonienRadchenko2024}. Dahne, G\'omez-Serrano, and Pech-Alberich recently proved the strict monotonicity for every $N\ge3$ \cite{DahneGomezSerranoPechAlberich2026}.
	
	The proof of Theorem \ref{thm:main} has two ingredients. First, joining the incenter to the vertices and tangency points decomposes $P$ into $2N$ right triangles. Dropping the transmission conditions across the internal cuts gives a mixed Dirichlet--Neumann relaxation, exact for the regular polygon by reflection symmetry. Second, we prove the strict concavity of
	\[
	\alpha\longmapsto h(\tan\alpha)-\frac18\tan\alpha
	\qquad\left(0<\alpha<\frac\pi2\right),
	\]
	with a uniform concavity modulus on compact subintervals, together with the sharp lower bound $h(a)\ge a/16$. The quantitative estimate \eqref{eq:main-deficit} then follows directly from Jensen's inequality and the fixed-area relation. The first term measures angular asymmetry; the second measures perimeter excess.
	
	The paper is organized as follows. Section \ref{sec:cells} establishes the mixed-cell relaxation. Section \ref{sec:concavity} proves the angular concavity estimate by Galerkin approximation and records an equivalent spectral-measure formulation. Section \ref{sec:mainproof} proves Theorem \ref{thm:main}. Section \ref{sec:monotonicity} proves Corollary \ref{cor:regular-monotonicity}, and Section \ref{sec:eigenvalue} gives the spectral consequences.
	
	\section{Tangential polygons and mixed cells}\label{sec:cells}
	
	Let $P$ be a tangential $N$-gon with incenter $O$ and inradius $r$. For each vertex, write its interior angle as $\pi-2\alpha_i$, where $0<\alpha_i<\pi/2$. Since the exterior angles sum to $2\pi$,
	\begin{equation}\label{eq:alphasum}
		\sum_{i=1}^N\alpha_i=\pi.
	\end{equation}
	The two tangent segments from the $i$th vertex to the incircle have common length $r\tan\alpha_i$. Hence, with
	\begin{equation}\label{eq:aiS}
		a_i=\tan\alpha_i,
		\qquad
		S=\sum_{i=1}^Na_i,
	\end{equation}
	we have
	\begin{equation}\label{eq:area-perimeter}
		|P|=r^2S,
		\qquad
		L(P)=2rS.
	\end{equation}
	
	For $a>0$, let
	\[
	D_a=\{(x,y)\in\R^2:0<x<1,\ 0<y<ax\}.
	\]
	Define
	\begin{equation}\label{eq:hdef}
		h(a)=\max_{\substack{v\in H^1(D_a)\\ \Tr v=0\text{ on }\{1\}\times(0,a)}}
		\left\{2\int_{D_a}v\,\dd x\dd y-\int_{D_a}|\nabla v|^2\,\dd x\dd y\right\}.
	\end{equation}
	The Euler--Lagrange equation has a Dirichlet condition on the vertical side and homogeneous Neumann conditions on the other two sides. By scaling, the corresponding functional on $rD_a$ equals $r^4h(a)$.
	
	Joining $O$ to every vertex and every tangency point decomposes $P$ into $2N$ right triangles. The two cells adjacent to the $i$th vertex are congruent to $rD_{a_i}$.
	
	\begin{proposition}[Mixed-cell relaxation]\label{prop:relax}
		Every tangential $N$-gon satisfies
		\begin{equation}\label{eq:relax}
			T(P)\le2r^4\sum_{i=1}^Nh(a_i).
		\end{equation}
		For a regular polygon, equality holds.
	\end{proposition}
	
	\begin{proof}
		The restriction of a function in $H_0^1(P)$ to each of the $2N$ cells vanishes on the part of the cell boundary lying in $\partial P$, and the restrictions have matching traces on the internal cuts. Removing these matching conditions enlarges the admissible space in \eqref{eq:variational-intro} to the direct product of the mixed Sobolev spaces on the cells. The resulting maximum is the sum of the cell maxima, which gives \eqref{eq:relax}.
		
		If $P$ is regular, every internal cut is contained in a reflection axis. By uniqueness, the torsion function is invariant under the corresponding reflection. Hence its weak normal derivative vanishes on each cut, and its restrictions solve the mixed problems on the cells. Equivalently, the mixed solution on one cell extends evenly across the cuts to the Dirichlet torsion function of the regular polygon. Thus the relaxed maximum is attained.
	\end{proof}
	
	\section{Strict angular concavity}\label{sec:concavity}
	
	The following angular concavity result is crucial in obtaining the quantitative P\'olya-Szeg\H{o} theorem for torsional rigidity.
	
	\begin{theorem}\label{thm:concavity}
		The map
		\[
		\alpha\longmapsto h(\tan\alpha)-\frac18\tan\alpha
		\]
		is strictly concave on $(0,\pi/2)$. More precisely, for every compact interval $J\Subset(0,\pi/2)$ there is $c_J>0$ such that
		\begin{align}\label{eq:strong-concavity}
			h\bigl(\tan(\vartheta\alpha+(1-\vartheta)\beta)\bigr)
			-\frac18\tan(\vartheta\alpha+(1-\vartheta)\beta)
			&\ge \vartheta\left(h(\tan\alpha)-\frac18\tan\alpha\right)\notag\\
			&\quad +(1-\vartheta)\left(h(\tan\beta)-\frac18\tan\beta\right)\notag\\
			&\quad+\frac{c_J}{2}\vartheta(1-\vartheta)(\alpha-\beta)^2
		\end{align}
		for all $\alpha,\beta\in J$ and $\vartheta\in[0,1]$. Moreover,
		\begin{equation}\label{eq:h-lower}
			h(a)\ge\frac a{16}
			\qquad(a>0).
		\end{equation}
	\end{theorem}
	
	The proof requires several preparatory steps.
	
	\subsection*{Step 1: Fixed-domain formulation and completion of the square}
	
	It is convenient to represent $h(a)$ as a maximization problem for a quadratic functional over a fixed domain.
	
	Set
	\begin{equation}\label{eq:Dfixed}
		D=\{(x,t)\in\R^2:0<t<x<1\},
		\qquad
		V=\{v\in H^1(D):\Tr v=0\text{ on }\{x=1\}\}.
	\end{equation}
	For $u,v\in V$, write
	\begin{equation}\label{eq:forms}
		I(v)=\int_Dv,
		\qquad
		X(u,v)=\int_Du_xv_x,
		\qquad
		Y(u,v)=\int_Du_tv_t,
	\end{equation}
	and abbreviate $X(v)=X(v,v)$ and $Y(v)=Y(v,v)$. The change of variables $y=at$ immediately implies the following formula:
	\begin{equation}\label{eq:hfixed}
		h(a)=\sup_{v\in V}\{2aI(v)-aX(v)-a^{-1}Y(v)\}.
	\end{equation}
	
	We can further simplify \eqref{eq:hfixed}. Note that for almost every $t$ and $t<x<1$,
	\[
	v(x,t)=-\int_x^1v_x(s,t)\,\dd s.
	\]
	Consequently,
	\begin{equation}\label{eq:poincare}
		\|v\|_{L^2(D)}^2\le X(v).
	\end{equation}
	Thus $X$ is an inner product on $V$, $I$ is $X$-continuous, and the norm $(X+Y)^{1/2}$ is equivalent to the usual $H^1(D)$ norm. By the Riesz representation theorem on the $X$-completion of $V$, there is a unique $w$ such that
	\begin{equation}\label{eq:Riesz}
		X(w,\varphi)=I(\varphi)
		\qquad(\forall \varphi\in V).
	\end{equation}
	In fact, $w\in V$ and is given explicitly by
	\begin{equation}\label{eq:wdef}
		w(x,t)=\frac{(1-t)^2-(x-t)^2}{2}.
	\end{equation}
	Indeed, $-w_{xx}=1$, $w(1,t)=0$, and $w_x(t,t)=0$, so integration by parts in $x$ gives \eqref{eq:Riesz}. In particular, direct computation yields
	\begin{equation}\label{eq:wenergy}
		I(w)=X(w)=\frac1{12}.
	\end{equation}
	Since for any $v\in V$,
	\begin{align*}
		X(v-w)=X(v)-2X(v,w)+X(w)=X(v)-2I(v)+\frac{1}{12},
	\end{align*}
	by \eqref{eq:hfixed}, we obtain
	\begin{equation}\label{eq:h-m}
		h(a)=\frac a{12}-m(a),
	\end{equation}
	where
	\begin{equation}\label{eq:mV}
		m(a)=\inf_{v\in V}\{aX(v-w)+a^{-1}Y(v)\}.
	\end{equation}
	This is the form used below.
	
	\subsection*{Step 2: Removing the zero mode}
	
	We next rewrite \eqref{eq:h-m} modulo the kernel of $Y$.
	
	Clearly, the kernel of $Y$ is
	\begin{equation}\label{eq:Kdef}
		\cK=\{k\in V:Y(k)=0\};
	\end{equation}
	its elements are precisely the functions depending only on $x$ and vanishing at $x=1$.
	
	To further simplify \eqref{eq:h-m}, we need the following lemma, which gives the projection formula for $v\in V$ onto $\mathcal{K}$.
	
	\begin{lemma}\label{lem:projection}
		For $v\in V$, define
		\begin{equation}\label{eq:PKdef}
			(P_{\cK}v)'(x)=\frac1x\int_0^xv_x(x,t)\,\dd t,
			\qquad
			(P_{\cK}v)(1)=0.
		\end{equation}
		Then $P_{\cK}v\in\cK$ and
		\begin{equation}\label{eq:PKorth}
			X(v-P_{\cK}v,k)=0
			\qquad(k\in\cK).
		\end{equation}
		Consequently,
		\begin{equation}\label{eq:decomp}
			V=\cK\oplus_X\cH,
			\qquad
			\cH:=\{q\in V:X(q,k)=0\text{ for all }k\in\cK\}.
		\end{equation}
	\end{lemma}
	
	\begin{proof}
		Jensen's inequality gives
		\begin{equation}\label{eq:PKenergy}
			\int_0^1x|(P_{\cK}v)'(x)|^2\,\dd x\le X(v).
		\end{equation}
		Moreover, since $(P_{\cK}v)(1)=0$,
		\[
		P_{\cK}v(x)=-\int_x^1(P_{\cK}v)'(s)\,\dd s,
		\]
		and hence
		\[
		|P_{\cK}v(x)|^2 \le \left(\int_x^1s|(P_{\cK}v)'(s)|^2\,\dd s\right) \left(\int_x^1 \frac{1}{s}\, \dd s\right)=\log\frac1x\int_x^1s|(P_{\cK}v)'(s)|^2\,\dd s.
		\]
		Multiplying by $x$ and applying Fubini's theorem shows that
		\[
		\int_D|P_{\cK}v|^2
		=\int_0^1x|P_{\cK}v(x)|^2\,\dd x
		\le C\int_0^1s|(P_{\cK}v)'(s)|^2\,\dd s<\infty.
		\]
		Together with \eqref{eq:PKenergy}, this proves that $P_{\cK}v\in H^1(D)$. Clearly,  $P_{\cK}v\in \cK$ because it is independent of $t$. Finally, for $k=k(x)\in\cK$,
		\[
		X(v-P_{\cK}v,k)
		=\int_0^1k'(x)
		\left[\int_0^xv_x(x,t)\,\dd t-x(P_{\cK}v)'(x)\right]\dd x=0.
		\]
		This proves the splitting statement in the lemma. Also, the uniqueness follows from the positive definiteness of $X$.
	\end{proof}
	
	The space $\cH$ is closed in the $(X+Y)^{1/2}$ norm, hence separable, and $Y$ is positive definite on $\cH$. Applying Lemma \ref{lem:projection} to $w$ gives
	\begin{equation}\label{eq:p-w1}
		p(x):=P_{\cK}w(x)=\frac{1-x^2}{4},
		\qquad
		w_1:=w-p\in\cH.
	\end{equation}
	Since $p$ and $w_1$ are $X$-orthogonal,
	\begin{equation}\label{eq:mass}
		X(w_1)=X(w)-X(p)=\frac1{48}.
	\end{equation}
	Writing $v=k+q$ and $w=p+w_1$ according to \eqref{eq:decomp}, the $\cK$-component of \eqref{eq:mV} is minimized at $k=p$. Therefore
	\begin{equation}\label{eq:mH}
		m(a)=\inf_{q\in\cH}\{aX(q-w_1)+a^{-1}Y(q)\}.
	\end{equation}
	
	The next lemma will be used in the proof of Theorem \ref{thm:concavity}.
	
	\begin{lemma}\label{lem:dual}
		For every $\varphi\in V$,
		\begin{equation}\label{eq:dualidentity}
			X(w_1,\varphi)=-\frac12\int_Dt\varphi_t,
		\end{equation}
		and
		\begin{equation}\label{eq:dualnorm}
			\sup_{Y(\varphi)>0}\frac{|X(w_1,\varphi)|^2}{Y(\varphi)}=\frac1{48}.
		\end{equation}
		The same supremum is obtained if $\varphi$ is restricted to $\cH$.
	\end{lemma}
	
	\begin{proof}
		Since $(w_1)_x=t-x/2$, the vector field
		\[
		\bm F(x,t)=\left(t-\frac x2,\frac t2\right)
		\]
		is divergence free and has zero normal component on the sides $t=0$ and $t=x$. Because $\varphi=0$ on $x=1$, the divergence theorem gives \eqref{eq:dualidentity}. Hence
		\[
		|X(w_1,\varphi)|^2
		\le\left(\int_D\frac{t^2}{4}\right)Y(\varphi)
		=\frac1{48}Y(\varphi).
		\]
		The constant is sharp. Indeed, choose $\chi_\varepsilon\in W^{1,\infty}(0,1)$ such that $0\le\chi_\varepsilon\le1$, $\chi_\varepsilon=1$ on $[0,1-\varepsilon]$, and $\chi_\varepsilon(1)=0$, and put
		\[
		\varphi_\varepsilon(x,t)=-\frac14\chi_\varepsilon(x)t^2.
		\]
		Then, by dominated convergence, $X(w_1,\varphi_\varepsilon)\to1/48$ and $Y(\varphi_\varepsilon)\to1/48$. 
		
		Finally, subtracting $P_{\cK}\varphi$ does not change either $X(w_1,\varphi)$ or $Y(\varphi)$, so the supremum may be taken over $\cH$.
	\end{proof}
	
	\subsection*{Step 3: Galerkin representation}
	
	Choose nested finite-dimensional spaces $H_n\subset\cH$ such that $w_1\in H_n$ and $\bigcup_nH_n$ is dense in $\cH$ for the $(X+Y)^{1/2}$ norm. Set
	\begin{equation}\label{eq:mn}
		m_n(a)=\min_{q\in H_n}\{aX(q-w_1)+a^{-1}Y(q)\}.
	\end{equation}
	Since the spaces are nested, $m_n(a)$ is decreasing. The quadratic functional in \eqref{eq:mn} is continuous in the $(X+Y)^{1/2}$ norm, and $\bigcup_n H_n$ is dense in $\cH$; consequently,
	\begin{equation}\label{eq:mnlimit}
		m_n(a)\downarrow m(a)
		\qquad(a>0).
	\end{equation}
	
	On $H_n$, $X$ is an inner product and $Y$ is a positive-definite symmetric form. Hence one can choose an $X$-orthonormal basis $e_{n,1},\dots,e_{n,d_n}$ and positive numbers $\lambda_{n,1},\dots,\lambda_{n,d_n}$ such that
	\begin{equation}\label{eq:diagonal}
		Y(e_{n,j},e_{n,k})=\lambda_{n,j}\delta_{jk}.
	\end{equation}
	Write $w_1=\sum_jc_{n,j}e_{n,j}$ and introduce the finite positive measure
	\begin{equation}\label{eq:mun}
		\mu_n=\sum_{j=1}^{d_n}c_{n,j}^2\delta_{\lambda_{n,j}}.
	\end{equation}
	Scalar minimization over $q=\sum_j s_{n,j} e_{n,j} \in H_n$ gives
	\begin{equation}\label{eq:mnmeasure}
		m_n(a)=\int_{(0,\infty)}\frac{a\lambda}{a^2+\lambda}\,\dd\mu_n(\lambda).
	\end{equation}
	Moreover, \eqref{eq:mass} and Lemma \ref{lem:dual} give the two moment bounds
	\begin{equation}\label{eq:moments-n}
		\int\dd\mu_n=\frac1{48},
		\qquad
		\int\lambda^{-1}\,\dd\mu_n(\lambda)
		=\sup_{0\ne q\in H_n}\frac{|X(w_1,q)|^2}{Y(q)}
		\le\frac1{48}.
	\end{equation}
	The identity in the second formula is the equality case of the weighted Cauchy--Schwarz inequality in the diagonal basis. More precisely, for any  $q=\sum_j s_{n,j} e_{n,j} \in H_n$, Cauchy--Schwarz yields
	\begin{align*}
		\frac{|X(w_1,q)|^2}{Y(q)}=\frac{(\sum_j c_{n,j}s_{n,j})^2}{\sum_j \lambda_{n,j}s_{n,j}^2}\le \sum_j \frac{c_{n,j}^2}{\lambda_{n,j}}=\int\lambda^{-1}\,\dd\mu_n(\lambda),
	\end{align*}
	and equality is attained by choosing $s_{n,j}=\tfrac{c_{n,j}}{\lambda_{n,j}}$.
	
	\subsection*{Step 4: The kernel inequality and proof of Theorem \ref{thm:concavity}}
	
	Now we are ready to prove Theorem \ref{thm:concavity}.

	\begin{proof}[Proof of Theorem \ref{thm:concavity}]
		For each $n$, consider the approximating function
		\[
		\alpha\longmapsto -\frac1{24}\tan\alpha-m_n(\tan\alpha).
		\]
		By \eqref{eq:h-m} and \eqref{eq:mnlimit}, these functions increase pointwise to
		\[
		\alpha\longmapsto h(\tan\alpha)-\frac18\tan\alpha.
		\]
		
		Let $a=\tan\alpha$. A direct differentiation gives
		\begin{equation}\label{eq:second-scalar}
			\frac{\dd^2}{\dd\alpha^2}
			\left(\frac{a\lambda}{a^2+\lambda}\right)
			=-a(1+a^2)K_{a^2}(\lambda).
		\end{equation}
		Hence, by \eqref{eq:mnmeasure} and the second derivative of $\tan\alpha$, the second derivative of the above approximant is
		\begin{equation}\label{eq:approximant-second}
			a(1+a^2)
			\left[\int K_{a^2}(\lambda)\,\dd\mu_n(\lambda)-\frac1{12}\right],
		\end{equation}
		where
		\begin{equation}\label{eq:kernel}
			K_x(\lambda):=\frac{2\lambda(3x\lambda-x-\lambda^2+3\lambda)}{(x+\lambda)^3},
		\end{equation}
		defined for $x,\lambda>0$.
		
		We shall use the following kernel inequality:
		\begin{equation}\label{eq:kernel-basic}
			K_x(\lambda)<4+\frac6{1+x}\left(\frac1\lambda-1\right).
		\end{equation}
		Indeed, the difference between the right- and left-hand sides is
		\begin{equation}\label{eq:remainder}
			R(x,\lambda)=\frac{2xP_x(\lambda)}{\lambda(x+\lambda)^3(1+x)},
		\end{equation}
		where
		\begin{align}\label{eq:poly}
			P_x(\lambda)={}&2\lambda x^3+(6\lambda^2-\lambda+3)x^2\notag\\
			&+(3\lambda^3-2\lambda^2+9\lambda)x
			+\lambda^2(3\lambda^2-9\lambda+10).
		\end{align}
		Each coefficient displayed in \eqref{eq:poly} is positive for $\lambda>0$: indeed,
		\[
		6\lambda^2-\lambda+3>0,\qquad
		3\lambda^2-2\lambda+9>0,\qquad
		3\lambda^2-9\lambda+10>0.
		\]
		Thus $P_x(\lambda)>0$ and hence $R(x,\lambda)>0$.
		
		More is true on compact $x$-intervals. If $0<x_0\le x\le x_1<\infty$, then
		\begin{equation}\label{eq:uniform-gap}
			\inf_{x\in[x_0,x_1],\,\lambda>0}R(x,\lambda)>0.
		\end{equation}
		Indeed,
		\[
		\lambda R(x,\lambda)\longrightarrow\frac6{1+x}
		\quad(\lambda\downarrow0),
		\qquad
		R(x,\lambda)\longrightarrow\frac{6x}{1+x}
		\quad(\lambda\to\infty),
		\]
		uniformly for $x\in[x_0,x_1]$; on the remaining compact rectangle the positive continuous function $R$ has a positive minimum.
		
		Fix $J=[\alpha_0,\alpha_1]\Subset(0,\pi/2)$ and let $\delta_J>0$ be the uniform lower bound in \eqref{eq:uniform-gap} for $x\in[\tan^2\alpha_0,\tan^2\alpha_1]$. From \eqref{eq:kernel-basic}--\eqref{eq:moments-n},
		\begin{align*}
			\int K_{a^2}(\lambda)\,\dd\mu_n(\lambda)
			&\le4\int\dd\mu_n
			+\frac6{1+a^2}
			\left(\int\lambda^{-1}\,\dd\mu_n-\int\dd\mu_n\right)
			-\delta_J\int\dd\mu_n\\
			&\le\frac1{12}-\frac{\delta_J}{48}.
		\end{align*}
		Consequently, the expression in \eqref{eq:approximant-second} is at most $-c_J$ on $J$, where
		\[
		c_J=\frac{\delta_J}{48}
		\min_{\alpha\in J}\tan\alpha\bigl(1+\tan^2\alpha\bigr)>0.
		\]
		Thus
		\[
		\alpha\longmapsto -\frac1{24}\tan\alpha-m_n(\tan\alpha)+\frac{c_J}{2}\alpha^2
		\]
		is concave on $J$. Pointwise limits preserve concavity, so passing to the limit gives \eqref{eq:strong-concavity}, and in particular strict concavity of $\alpha\mapsto h(\tan\alpha)-\frac18\tan\alpha$.
		
		Finally, the trial function $p(x)=(1-x^2)/4$ satisfies $Y(p)=0$ and $I(p)=X(p)=1/16$. Substitution into \eqref{eq:hfixed} gives \eqref{eq:h-lower}.
	\end{proof}

	\begin{remark}[Sharpness of \eqref{eq:h-lower}]\label{rem:h-lower-sharp}
		The constant $1/16$ in \eqref{eq:h-lower} is optimal. More precisely,
		\begin{equation}\label{eq:h-lower-sharp}
			\inf_{a>0}\frac{h(a)}{a}=\frac1{16}.
		\end{equation}
		To see this, let $R_N$ be the regular $N$-gon of area $\pi$ and let $r_N$ be its inradius. By the exactness statement in Proposition~\ref{prop:relax},
		\[
		T(R_N)=2Nr_N^4h\left(\tan\frac{\pi}{N}\right),
		\]
		while \eqref{eq:area-perimeter} gives
		\[
		\pi=Nr_N^2\tan\frac{\pi}{N}.
		\]
		Consequently,
		\begin{equation}\label{eq:h-regular-cell}
			\frac{h(\tan(\pi/N))}{\tan(\pi/N)}
			=\frac{N\tan(\pi/N)}{2\pi^2}T(R_N).
		\end{equation}
		The Saint--Venant inequality yields $T(R_N)\le \pi/8$, since the disk of area $\pi$ has torsional rigidity $\pi/8$. Combining this with \eqref{eq:h-lower},
		\[
		\frac1{16}
		\le \frac{h(\tan(\pi/N))}{\tan(\pi/N)}
		\le \frac{N\tan(\pi/N)}{16\pi}.
		\]
		Since $N\tan(\pi/N)\to\pi$, the squeeze theorem gives
		\[
		\frac{h(\tan(\pi/N))}{\tan(\pi/N)}\longrightarrow\frac1{16},
		\]
		which proves \eqref{eq:h-lower-sharp}. Notice that this argument does not require the monotonicity of $T(R_N)$.
		
		A second proof connects the sharp constant directly with the mixed-cell asymptotics in \cite[Theorem~1.3]{GuiHuLiZhang2026}. Let
		\[
		\Omega_q=\{(x,y):0<x<q,\ 0<y<1-x/q\},
		\]
		with a Dirichlet condition on $\{x=0\}$ and homogeneous Neumann conditions on the other two sides. After dilation by $a^{-1}$ and a rigid reflection, $D_a$ is mapped onto $\Omega_{1/a}$, with the boundary conditions preserved. Hence the fourth-order scaling of torsional rigidity gives
		\begin{equation}\label{eq:h-omega-scaling}
			h(a)=a^4T(\Omega_{1/a}).
		\end{equation}
		The asymptotic expansion from \cite[Theorem~1.3]{GuiHuLiZhang2026} is
		\[
		T(\Omega_q)=\frac{q^3}{16}+\frac{q}{48}-\frac{\zeta(3)}{2\pi^3}+O(q^{-1})
		\qquad(q\to\infty).
		\]
		Substituting $q=1/a$ into \eqref{eq:h-omega-scaling} yields
		\begin{equation}\label{eq:h-small-a}
			h(a)=\frac a{16}+\frac{a^3}{48}-\frac{\zeta(3)}{2\pi^3}a^4+O(a^5)
			\qquad(a\downarrow0).
		\end{equation}
		Thus $h(a)/a\to1/16$. This gives the sharpness from a different viewpoint and, in addition, the rate at which the lower bound is approached.
	\end{remark}
	
	\subsection*{An equivalent spectral-measure formulation}\label{subsec:spectral}
	
	The Galerkin measures $\mu_n$ above are finite-dimensional substitutes for a single spectral measure. We record the equivalent infinite-dimensional formulation because it isolates the operator-theoretic structure behind the concavity argument.
	
	On the $X$-completion of $\cH$, with form domain $\cH$, the form $Y$ is densely defined, symmetric, nonnegative, and closed, because its form norm is $(X+Y)^{1/2}$ and $\cH$ is closed in that norm. Kato's second representation theorem \cite[Chapter~VI, Theorem~2.23]{Kato1976} therefore associates with $Y$ a nonnegative self-adjoint operator $A$ satisfying
	\begin{equation}\label{eq:kato}
		D(A^{1/2})=\cH,
		\qquad
		Y(q,\varphi)=X(A^{1/2}q,A^{1/2}\varphi).
	\end{equation}
	No coercivity assumption is needed. Since $Y$ is positive definite on $\cH$ after the zero mode has been removed, $\ker A=\{0\}$, although $0$ may still belong to the spectrum.
	
	Let $E$ be the spectral resolution of $A$ and define
	\begin{equation}\label{eq:mu}
		\mu(B)=X(E(B)w_1,w_1),
		\qquad B\subset(0,\infty)\ \text{Borel}.
	\end{equation}
	The spectral theorem and the same scalar minimization as in \eqref{eq:mnmeasure} give
	\begin{equation}\label{eq:spectral-m}
		m(a)=\int_{(0,\infty)}\frac{a\lambda}{a^2+\lambda}\,\dd\mu(\lambda).
	\end{equation}
	Moreover,
	\begin{equation}\label{eq:spectral-moments}
		\int\dd\mu=X(w_1)=\frac1{48},
		\qquad
		\int\lambda^{-1}\,\dd\mu(\lambda)
		=\sup_{Y(q)>0}\frac{|X(w_1,q)|^2}{Y(q)}
		=\frac1{48}.
	\end{equation}
	The second identity is the spectral characterization of the dual form norm: after truncation to $[1/n,n]$, Cauchy--Schwarz gives one inequality, while the test function $A^{-1}E([1/n,n])w_1$ gives the reverse one; then $n\to\infty$. Integrating the kernel inequality \eqref{eq:kernel-basic} against $\mu$ therefore proves the strict concavity directly. This is the infinite-dimensional counterpart of the Galerkin proof; see \cite{Kato1976,ReedSimon1980} for the underlying form and spectral theory.
	
	\section{Proof of the quantitative inequality}\label{sec:mainproof}
	
	\begin{proof}[Proof of Theorem \ref{thm:main}]
		By the strict convexity of $\tan$ and \eqref{eq:alphasum},
		\begin{equation}\label{eq:S-lower}
			S=\sum_{i=1}^N\tan\alpha_i\ge Na_*.
		\end{equation}
		Set
		\begin{equation}\label{eq:zdef}
			z:=\frac{S}{Na_*}\ge1.
		\end{equation}
		Let $r_N$ be the inradius of the equal-area regular $N$-gon $R_N$. Since
		\[
		|P|=r^2Na_*z,
		\qquad
		|R_N|=r_N^2Na_*,
		\]
		we have
		\begin{equation}\label{eq:rN}
			r_N^2=r^2z,
			\qquad
			T(R_N)=2Nr^4z^2h(a_*).
		\end{equation}
		Moreover, by \eqref{eq:area-perimeter},
		\[
		L(P)=2rNa_*z,
		\qquad
		L(R_N)=2r_NNa_*=2r\sqrt z\,Na_*,
		\]
		so
		\begin{equation}\label{eq:z-perimeter}
			z=\left(\frac{L(P)}{L(R_N)}\right)^2.
		\end{equation}
		
		Theorem \ref{thm:concavity} and $\sum_i\alpha_i=\pi$ give
		\begin{equation}\label{eq:jensen-direct}
			\sum_{i=1}^N\left(h(a_i)-\frac{a_i}{8}\right)
			\le
			N\left(h(a_*)-\frac{a_*}{8}\right).
		\end{equation}
		By strict concavity, equality in \eqref{eq:jensen-direct} holds if and only if $\alpha_i=\pi/N$ for every $i$. Likewise, strict convexity of $\tan$ shows that equality in \eqref{eq:S-lower}, equivalently $z=1$, holds if and only if $\alpha_i=\pi/N$ for every $i$. In that case every side has length $r(\tan\alpha_i+\tan\alpha_{i+1})=2r\tan(\pi/N)$, so $P$ is regular. Hence each of the two nonnegative terms on the right-hand side of \eqref{eq:main-deficit} vanishes if and only if $P$ is regular.
		Combining Proposition \ref{prop:relax} with $S=Na_*z$ yields
		\begin{align}
			T(R_N)-T(P)
			&\ge 2r^4\left[Nz^2h(a_*)-\sum_{i=1}^Nh(a_i)\right]\notag\\
			&=2r^4\Bigg[
			N\left(h(a_*)-\frac{a_*}{8}\right)
			-\sum_{i=1}^N\left(h(a_i)-\frac{a_i}{8}\right)\notag\\
			&\hspace{25mm}
			+N(z-1)\left(h(a_*)(z+1)-\frac{a_*}{8}\right)
			\Bigg]\notag\\
			&\ge 2r^4\left[
			N\left(h(a_*)-\frac{a_*}{8}\right)
			-\sum_{i=1}^N\left(h(a_i)-\frac{a_i}{8}\right)
			+\frac{Na_*}{16}(z-1)^2
			\right].\label{eq:master-deficit}
		\end{align}
		The last inequality uses $h(a_*)\ge a_*/16$. Substituting \eqref{eq:z-perimeter} gives \eqref{eq:main-deficit}; its right-hand side is nonnegative by \eqref{eq:jensen-direct}.
		
		If equality holds in $T(P)\le T(R_N)$, then \eqref{eq:master-deficit} and the preceding equality characterization force $P$ to be regular. Conversely, equality holds for the regular polygon by Proposition \ref{prop:relax}.
		
		Finally, suppose that $\alpha_i\in J$ for every $i$. The strong concavity estimate \eqref{eq:strong-concavity} is equivalent to concavity on $J$ of
		\[
		\alpha\longmapsto h(\tan\alpha)-\frac18\tan\alpha+\frac{c_J}{2}\alpha^2.
		\]
		Applying Jensen's inequality to this map and using $\sum_i\alpha_i=\pi$ gives
		\begin{align*}
			&N\left(h(a_*)-\frac{a_*}{8}\right)
			-\sum_{i=1}^N\left(h(a_i)-\frac{a_i}{8}\right)\\
			&\qquad\ge \frac{c_J}{2}
			\left(\sum_{i=1}^N\alpha_i^2-N\left(\frac\pi N\right)^2\right)
			=\frac{c_J}{2}\sum_{i=1}^N\left(\alpha_i-\frac\pi N\right)^2.
		\end{align*}
		Together with \eqref{eq:master-deficit} and \eqref{eq:z-perimeter}, this proves \eqref{eq:main-stability}.
	\end{proof}
	
	\begin{proof}[Proof of Corollary \ref{cor:perimeter-deficit}]
		Dropping the nonnegative angular term from \eqref{eq:main-deficit} gives \eqref{eq:perimeter-deficit-r}. With $z$ as in \eqref{eq:zdef}, we have $A=r^2Na_*z$ and, by \eqref{eq:z-perimeter}, $z=L(P)^2/L(R_N)^2$. Hence
		\[
		\frac{Na_*r^4}{8}(z-1)^2
		=\frac{A^2}{8Na_*}\left(1-\frac1z\right)^2
		=\frac{A^2}{8N\tan(\pi/N)}
		\left(1-\frac{L(R_N)^2}{L(P)^2}\right)^2,
		\]
		which proves \eqref{eq:perimeter-deficit-area}.
	\end{proof}
	
	The scale-invariant consequence is
	\begin{equation}\label{eq:normalized}
		\frac{T(P)}{|P|^2}
		\le
		\frac{2h(\tan(\pi/N))}{N\tan^2(\pi/N)},
	\end{equation}
	with equality only for the regular polygon.
	
	\section{Monotonicity along regular polygons}\label{sec:monotonicity}
	
	We use the perimeter deficit in Corollary \ref{cor:perimeter-deficit} to prove Corollary \ref{cor:regular-monotonicity}.
	
	\begin{proof}[Proof of Corollary \ref{cor:regular-monotonicity}]
		Fix an area $A>0$ and an integer $N\ge3$. For $\varepsilon>0$ sufficiently small, set
		\[
		\alpha_1^\varepsilon=\varepsilon,
		\qquad
		\alpha_2^\varepsilon=\cdots=\alpha_{N+1}^\varepsilon
		=\frac{\pi-\varepsilon}{N}.
		\]
		These parameters are positive and satisfy $\sum_{i=1}^{N+1}\alpha_i^\varepsilon=\pi$. Choose successive tangent lines to a circle so that their turning angles are $2\alpha_i^\varepsilon$. Since the total turning angle is $2\pi$, these lines determine a tangential $(N+1)$-gon $P_\varepsilon$, which we scale so that $|P_\varepsilon|=A$.
		
		As $\varepsilon\downarrow0$, the interior angle corresponding to $\alpha_1^\varepsilon$ tends to $\pi$, the two adjacent sides merge, and all the remaining angle parameters tend to $\pi/N$. Hence, after a rigid motion, $P_\varepsilon\to R_N$ in the Hausdorff topology. Place $R_N$ with its incenter at the origin. For convex bodies, Hausdorff convergence is equivalent to uniform convergence of support functions; since the support function of $R_N$ is strictly positive, for every $\delta\in(0,1)$ and all sufficiently small $\varepsilon$,
		\[
		(1-\delta)R_N\subset P_\varepsilon\subset(1+\delta)R_N.
		\]
		By monotonicity and fourth-order scaling of torsional rigidity,
		\[
		(1-\delta)^4T(R_N)\le T(P_\varepsilon)\le(1+\delta)^4T(R_N).
		\]
		Letting first $\varepsilon\downarrow0$ and then $\delta\downarrow0$ gives
		\begin{equation}\label{eq:torsion-limit}
			T(P_\varepsilon)\longrightarrow T(R_N).
		\end{equation}
		
		Let $r_\varepsilon$ be the inradius of $P_\varepsilon$. 
		By \eqref{eq:z-perimeter}, applied with $N$ replaced by $N+1$,
		\[
		z_\varepsilon
		:=
		\left(\frac{L(P_\varepsilon)}{L(R_{N+1})}\right)^2
		=
		\frac{\tan\varepsilon
			+N\tan((\pi-\varepsilon)/N)}
		{(N+1)\tan(\pi/(N+1))}.
		\]
		Since
		\[
		A=r_\varepsilon^2\left[\tan\varepsilon+N\tan\frac{\pi-\varepsilon}{N}\right],
		\]
		we have
		\begin{equation}\label{eq:repsilon-limit}
			r_\varepsilon^2\longrightarrow
			\frac{A}{N\tan(\pi/N)}=:r_N^2,
		\end{equation}
		where $r_N$ is the inradius of the area-$A$ regular $N$-gon.
		
		Apply Corollary \ref{cor:perimeter-deficit} in the form \eqref{eq:perimeter-deficit-r} to $P_\varepsilon$ and the equal-area regular $(N+1)$-gon $R_{N+1}$. This gives
		\begin{equation}\label{eq:monotonicity-remainder}
			T(R_{N+1})-T(P_\varepsilon)
			\ge
			\frac{r_\varepsilon^4(N+1)\tan(\pi/(N+1))}{8}
			(z_\varepsilon-1)^2.
		\end{equation}
		Moreover,
		\[
		z_\varepsilon\longrightarrow
		\frac{N\tan(\pi/N)}{(N+1)\tan(\pi/(N+1))}>1,
		\]
		because $x\mapsto\tan x/x$ is strictly increasing on $(0,\pi/2)$ and $\pi/N>\pi/(N+1)$. Passing to the limit in \eqref{eq:monotonicity-remainder}, and using \eqref{eq:torsion-limit} and \eqref{eq:repsilon-limit}, we obtain
		\begin{align}\label{eq:explicit-monotonicity}
			T(R_{N+1})-T(R_N)
			&\ge
			\frac{r_N^4(N+1)\tan(\pi/(N+1))}{8}\notag\\
			&\quad\times
			\left(
			\frac{N\tan(\pi/N)}{(N+1)\tan(\pi/(N+1))}-1
			\right)^2>0.
		\end{align}
		This proves the strict monotonicity.
	\end{proof}

	\section{Spectral consequence}\label{sec:eigenvalue}
	
	Let $\lambda_1(\Omega)$ denote the first Dirichlet eigenvalue of $-\Delta$ on $\Omega$, and let $j_{0,1}$ be the first positive zero of the Bessel function $J_0$. The planar Kohler--Jobin inequality states that
	\begin{equation}\label{eq:KJ}
		\lambda_1(\Omega)^2T(\Omega)\ge \frac{\pi j_{0,1}^4}{8},
	\end{equation}
	with equality if and only if $\Omega$ is a disk; see, for instance, \cite{Brasco2014}.
	
	Let $A>0$, and for each $N\ge3$ let $R_N$ be the regular $N$-gon of area $A$. Define the Kohler--Jobin slack of $R_N$ by
	\begin{equation}\label{eq:KJ-gap}
		\eta_N(A):=T(R_N)-\frac{\pi j_{0,1}^4}{8\lambda_1(R_N)^2}.
	\end{equation}
	Since $R_N$ is not a disk, the strict case of \eqref{eq:KJ} gives $\eta_N(A)>0$.
	
	Theorem \ref{thm:main} gives a direct spectral criterion. Recall that $a_*=\tan(\pi/N)$.
	
	\begin{proposition}\label{prop:eigenvalue-consequence}
		Let $P$ be a tangential $N$-gon satisfying $|P|=|R_N|=A$. If
		\begin{equation}\label{eq:eigenvalue-criterion}
			2r^4\left[
			N\left(h(a_*)-\frac{a_*}{8}\right)
			-\sum_{i=1}^N\left(h(a_i)-\frac{a_i}{8}\right)
			+\frac{Na_*}{16}\left(\frac{L(P)^2}{L(R_N)^2}-1\right)^2
			\right]
			\ge\eta_N(A),
		\end{equation}
		then
		\[
		\lambda_1(P)>\lambda_1(R_N).
		\]
	\end{proposition}
	
	\begin{proof}
		By Theorem \ref{thm:main} and \eqref{eq:eigenvalue-criterion},
		\[
		T(P)\le T(R_N)-\eta_N(A)
		=\frac{\pi j_{0,1}^4}{8\lambda_1(R_N)^2}.
		\]
		Since a polygon cannot be a disk, the Kohler--Jobin inequality is strict for $P$. Consequently,
		\[
		\lambda_1(P)^2
		>\frac{\pi j_{0,1}^4}{8T(P)}
		\ge\lambda_1(R_N)^2,
		\]
		which proves the claim.
	\end{proof}
	
	The perimeter-only estimate \eqref{eq:perimeter-deficit-area} immediately gives the simpler sufficient condition
	\begin{equation}\label{eq:perimeter-eigenvalue-criterion}
		\frac{A^2}{8N\tan(\pi/N)}
		\left(1-\frac{L(R_N)^2}{L(P)^2}\right)^2
		\ge \eta_N(A)
		\quad\Longrightarrow\quad
		\lambda_1(P)>\lambda_1(R_N).
	\end{equation}
	Hence every possible counterexample within the tangential class must satisfy the strict reverse of \eqref{eq:perimeter-eigenvalue-criterion}.
	
	We next determine the large-$N$ behavior of the threshold $\eta_N(A)$.
	
	\begin{lemma}[Asymptotics of the Kohler--Jobin slack]\label{lem:KJ-gap-asymptotic}
		As $N\to\infty$,
		\begin{equation}\label{eq:eta-asymptotic}
			\eta_N(A)=\frac{\pi^3A^2}{45N^4}
			+O\left(\frac{A^2}{N^5}\right).
		\end{equation}
	\end{lemma}
	
	\begin{proof}
		Let $\widehat R_N$ denote the regular $N$-gon of area $\pi$. The regular-polygon torsion expansion obtained in \cite[Theorem~1.4]{GuiHuLiZhang2026} is
		\begin{equation}\label{eq:torsion-regular-expansion-pi}
			T(\widehat R_N)
			=\frac{\pi}{8}-\frac{\pi\zeta(3)}{N^3}
			+\frac{\pi^5}{45N^4}+O(N^{-5}).
		\end{equation}
		On the other hand, the first Dirichlet eigenvalue expansion of Berghaus, Georgiev, Monien, and Radchenko \cite[Equation~(2) and Table~1]{BerghausGeorgievMonienRadchenko2024} gives
		\begin{equation}\label{eq:eigenvalue-regular-expansion-pi}
			\lambda_1(\widehat R_N)
			=j_{0,1}^2\left(1+\frac{4\zeta(3)}{N^3}+O(N^{-5})\right).
		\end{equation}
		In particular, the coefficient of $N^{-4}$ in the eigenvalue expansion vanishes.
		
		Since $R_N=\sqrt{A/\pi}\,\widehat R_N$, the scaling laws for torsional rigidity and the first Dirichlet eigenvalue yield
		\begin{align}
			T(R_N)
			&=\left(\frac{A}{\pi}\right)^2T(\widehat R_N)\notag\\
			&=\frac{A^2}{8\pi}-\frac{A^2\zeta(3)}{\pi N^3}
			+\frac{\pi^3A^2}{45N^4}
			+O\left(\frac{A^2}{N^5}\right),
			\label{eq:torsion-regular-expansion-A}
		\end{align}
		and
		\begin{equation}\label{eq:eigenvalue-regular-expansion-A}
			\lambda_1(R_N)
			=\frac{\pi j_{0,1}^2}{A}
			\left(1+\frac{4\zeta(3)}{N^3}+O(N^{-5})\right).
		\end{equation}
		Therefore,
		\begin{align}
			\frac{\pi j_{0,1}^4}{8\lambda_1(R_N)^2}
			&=\frac{A^2}{8\pi}
			\left(1+\frac{4\zeta(3)}{N^3}+O(N^{-5})\right)^{-2}\notag\\
			&=\frac{A^2}{8\pi}-\frac{A^2\zeta(3)}{\pi N^3}
			+O\left(\frac{A^2}{N^5}\right).
			\label{eq:KJ-comparison-expansion}
		\end{align}
		Subtracting \eqref{eq:KJ-comparison-expansion} from \eqref{eq:torsion-regular-expansion-A} proves \eqref{eq:eta-asymptotic}.
	\end{proof}
	
	We now derive the resulting asymptotic restriction on possible counterexamples. Suppose that $P$ is a tangential $N$-gon such that
	\[
	|P|=|R_N|=A
	\qquad\text{and}\qquad
	\lambda_1(P)\le \lambda_1(R_N).
	\]
	Then \eqref{eq:perimeter-eigenvalue-criterion} cannot hold, and hence
	\begin{equation}\label{eq:reverse-perimeter-condition}
		\frac{A^2}{8N\tan(\pi/N)}
		\left(1-\frac{L(R_N)^2}{L(P)^2}\right)^2
		<\eta_N(A).
	\end{equation}
	Equivalently,
	\begin{equation}\label{eq:counterexample-first-bound-exact}
		1-\frac{L(R_N)^2}{L(P)^2}
		<\left(\frac{8N\tan(\pi/N)}{A^2}\eta_N(A)\right)^{1/2}.
	\end{equation}
	By Lemma \ref{lem:KJ-gap-asymptotic} and
	\[
	N\tan\frac{\pi}{N}=\pi+O(N^{-2}),
	\]
	we obtain
	\begin{equation}\label{eq:counterexample-first-bound}
		1-\frac{L(R_N)^2}{L(P)^2}
		<\pi^2\sqrt{\frac{8}{45}}\,\frac1{N^2}+O(N^{-3}).
	\end{equation}
	Furthermore, $L(P)\ge L(R_N)$ for tangential $N$-gons of the same area, because \eqref{eq:z-perimeter} gives $L(P)/L(R_N)=\sqrt z\ge1$. Since \eqref{eq:counterexample-first-bound} shows that $1-L(R_N)^2/L(P)^2=O(N^{-2})$,
	\begin{align*}
		\frac{L(P)}{L(R_N)}
		&=\left(1-\left(1-\frac{L(R_N)^2}{L(P)^2}\right)\right)^{-1/2}\\
		&=1+\frac12\left(1-\frac{L(R_N)^2}{L(P)^2}\right)+O(N^{-4}).
	\end{align*}
	Consequently, every possible counterexample satisfies
	\begin{equation}\label{eq:counterexample-perimeter-ratio-bound}
		\frac{L(P)}{L(R_N)}-1
		<\frac{\pi^2}{2}\sqrt{\frac{8}{45}}\,\frac1{N^2}+O(N^{-3}).
	\end{equation}
	The remainder in \eqref{eq:counterexample-perimeter-ratio-bound} can be taken uniformly over all possible counterexamples $P$ (for fixed $A$): the right-hand side of \eqref{eq:counterexample-first-bound-exact} is independent of $P$, and the Taylor expansion above is uniform when $1-L(R_N)^2/L(P)^2=O(N^{-2})$.
	
	\begin{corollary}[Asymptotic exclusion of counterexamples away from the regular polygon]\label{cor:spectral-exclusion}
		Let
		\[
		C_*:=\frac{\pi^2}{2}\sqrt{\frac{8}{45}}.
		\]
		For every $\varepsilon>0$, there exists $N_\varepsilon\ge3$ such that, for every $N\ge N_\varepsilon$ and every tangential $N$-gon $P$ satisfying
		\[
		|P|=|R_N|=A,
		\]
		the condition
		\[
		\frac{L(P)}{L(R_N)}-1
		\ge\frac{C_*+\varepsilon}{N^2}
		\]
		implies
		\[
		\lambda_1(P)>\lambda_1(R_N).
		\]
		Equivalently, every possible counterexample to the polygonal Faber--Krahn inequality within the class of tangential $N$-gons must satisfy
		\[
		\frac{L(P)}{L(R_N)}-1
		<\frac{C_*+\varepsilon}{N^2}
		\]
		for all sufficiently large $N$.
	\end{corollary}
	
	\begin{proof}
		Suppose that $P$ is a possible counterexample, that is,
		\[
		\lambda_1(P)\le\lambda_1(R_N).
		\]
		By \eqref{eq:counterexample-perimeter-ratio-bound},
		\[
		\frac{L(P)}{L(R_N)}-1
		<\frac{C_*}{N^2}+O(N^{-3})
		\qquad\text{as }N\to\infty.
		\]
		By the uniformity just noted, for every fixed $\varepsilon>0$ the remainder $O(N^{-3})$ is bounded above by $\varepsilon N^{-2}$ once $N$ is sufficiently large, uniformly over all such $P$. Hence
		\[
		\frac{L(P)}{L(R_N)}-1
		<\frac{C_*+\varepsilon}{N^2}.
		\]
		The asserted implication follows by contraposition.
	\end{proof}
	
	Thus, for large $N$, every possible counterexample within the tangential class lies in an explicit relative perimeter layer of width $O(N^{-2})$ around $R_N$.
	
	\section*{Statement on the use of AI tools}
	During the preparation of this manuscript, the authors used ChatGPT (OpenAI) for language and presentation assistance and as an interactive tool for exploratory mathematical discussion. These interactions served to explore and refine possible lines of argument. All mathematical statements, proofs, and computations in the final manuscript were independently verified by the authors, who take full responsibility for the article.

\end{document}